\documentclass[12pt]{amsart}
\usepackage{amssymb}
\usepackage{latexsym}
\usepackage{amsfonts}
\usepackage{amsmath}
\usepackage{amssymb}

\newtheorem{theorem}{Theorem}[section]

\newtheorem{lemma}[theorem]{Lemma}

\theoremstyle{definition}

\def\la{{\langle}}
\def\ra{{\rangle}}
\def\char{{\rm char \,}}

\renewcommand{\leq}{\leqslant}
\renewcommand{\geq}{\geqslant}

\newcommand{\aut}[1]{{\sf Aut}(#1)}

\newcommand{\centr}[1]{{\sf Z}(#1)}

\newcommand{\diag}[1]{{\sf diag}(#1)}

\newcommand{\gl}[2]{{\sf GL}(#1,#2)}

\newcommand{\F}{\mathbb F}
\newcommand{\K}{\mathbb K}

\newcommand{\res}[1]{#1^{[p]}}

\newcommand{\ad}{{\sf ad}\,}
\newcommand{\ress}[1]{#1^{[2]}}
\newcommand{\resss}[1]{#1^{[3]}}

\begin{document}

\title[Small dimensional restricted  Lie algebras]
{The classification of $p$-nilpotent restricted  Lie algebras of dimension at most 4}
\author{\textsc{Csaba Schneider}}
\address{Departamento de Matemática,
Instituto de Ciencias Exatas,
Universidade Federal de Minas Gerais,
Av. Antonio Carlos 6627,  31270-901, Belo Horizonte, MG, Brazil}
\email{csaba.schneider@gmail.com}

%\thanks{The first author was supported by the FCT project
%PTDC/MAT/101993/2008 and by the Hungarian
%Scientific Research Fund (OTKA) grant~72845. 
%The second author was supported by  NSERC}

\author{\textsc{Hamid Usefi}}
\address{Department of Mathematics and Statistics,
Memorial University of Newfoundland,
St. John's, NL,
Canada, 
A1C 5S7}
\email{usefi@mun.ca}

\date{\today}

\thanks{The second author was supported by  NSERC}

\begin{abstract}
In this paper we obtain the classification of $p$-nilpotent restricted Lie algebras of dimension at most four
over a perfect field of characteristic $p$.
\end{abstract}

%\date{\today}

\maketitle

\section{Introduction}\label{intro}
In this paper we initiate the classification of small dimensional restricted Lie algebras. Similar classifications for
ordinary Lie algebras have a long history. The classification of all nilpotent Lie algebras up to dimension five over any  field  has been known for a long time. However, in dimension 6, the characterization depends on the underlying field. In 1958  Morozov \cite{morozov} gave a classification of nilpotent Lie algebras of dimension 6 over a filed of characteristic zero, see also   \cite{BK, nielsen, gong}  for a classification over other fields. 
These classifications, however, differ and it was not easy to compare them until recently that de Graaf \cite{dg} gave a complete classification over any field of characteristic other than 2.
de Graaf's approach  can be  verified computationally and was later revised and extended to characteristic 2
in \cite{CDS}. The classification in dimensions more than 6 is still in progress, see for example \cite{See, Ro}.

In this paper we give a list of $[p]$-nilpotent restricted Lie algebras of dimension at most 4 over a perfect field of characteristic $p\geq 3$. Let $L$ be a Lie algebra over a field $\F$ of characteristic $p$. Recall  
that $L$ is said to be  {\em restrictable}
 if $L$ affords a $[p]$-map $x\mapsto \res x$  that satisfies
the following properties:
\begin{enumerate}
\item $(\ad a)^p=\ad\res a$;
\item $\res{(\alpha a)}=\alpha^p\res a$;
\item $\res{(a+b)}=\res a+\res b+\sum_{i=1}^{p-1}s_i(a,b)$ 
\end{enumerate}
where $s_i(a,b)$ is given by the formula
$$
\left(\ad(a\otimes X+b\otimes 1)\right)^{p-1}(a\otimes 1)=
\sum_{i=1}^{p-1}is_i(a,b)\otimes X^{i-1}
$$
interpreted in the Lie algebra $L\otimes\F[X]$ over the polynomial 
ring $\F[X]$. A Lie algebra $L$ with a given $[p]$-map $x\mapsto \res x$ is 
said to be {\em restricted}.

Recall that a restricted Lie algebra $L$ is called $[p]$-nilpotent 
if there exists an integer $n$ such that $L^{[p]^n}=0$.
Let $L$ be a finite-dimensional $[p]$-nilpotent restricted Lie algebra.
Then, by Engel's Theorem, $L$ is nilpotent. 
Note that $(x+y)^{[p]}=\res x+ \res y$ modulo $\gamma_p(L)$, 
for every $x, y\in L$. Here, $\gamma_i(L)$ is the $i$-th term of the lower central series of $L$.
So if the nilpotency class  of $L$ is smaller than $p$, then 
 the $[p]$-map  is 
a semilinear transformation from $L$ to the center $\centr L$ of $L$. 
Let $\varphi_1,\ \varphi_2:L\rightarrow\centr L$ be two semilinear 
transformations. Then the restricted Lie algebras $(L,\varphi_1)$ and 
$(L,\varphi_2)$ are isomorphic if and only if there exists
$A\in\aut L$ such that 
$$
x\varphi_1A=xA\varphi_2\quad\mbox{holds for all}\quad x\in L.
$$
Hence, $\varphi_1$ and $\varphi_2$ define isomorphic restricted Lie algebras 
if and only if  there exists 
$A\in\aut L$ such that 
$A\varphi_1A^{-1}=\varphi_2$; that is, they are conjugate under
the automorphism group of $L$.
In this case we say that the $[p]$-maps $\varphi_1$ and $\varphi_2$ are
{\em equivalent}.
This defines a left action of $\aut L$ on the set of $[p]$-maps and the
isomorphism classes of restricted Lie algebras correspond to 
the $\aut L$-orbits under this action.

Our work is motivated by the isomorphism problem for  enveloping 
algebras of restricted Lie algebras. We are interested in understanding when two
non-isomorphic restricted Lie algebras can have isomorphic restricted
enveloping algebras and it makes sense to examine the class of $[p]$-nilpotent 
restricted Lie algebras. The first step towards this is a classification
of such restricted Lie algebras in small dimensions.

\section{The main result}

The main theorem of the paper is a classification of $[p]$-nilpotent 
restricted Lie algebra with dimension at most 4 over
perfect fields $\F$. We use as our starting 
point, the classification of nilpotent Lie algebras of dimension 4 and
classify the possible equivalence classes of $[p]$-maps on these Lie algebras.

In dimensions 1 and 2, 
there is a unique isomorphism type of nilpotent Lie algebras.
There are two nilpotent Lie algebras of dimension 3, 
one is abelian and the other one is nilpotent of class 2. 
There are three nilpotent Lie algebras of dimension 4, one is abelian, one 
nilpotent of class 2, and one nilpotent of class 3. If $p$ is greater than
the nilpotency class, then a $[p]$-map is semilinear, this is however
not always the case  in characteristic 3 and 5.

Our main result assumes that the field is perfect. The reason for this is 
that we often need that the Frobenius automorphism $x\mapsto x^p$ 
of $\F$ is surjective. If $\char\F=2$, then we denote by $\K$ the
Artin--Schreier subspace $\K=\{\delta+\delta^2\mid\delta\in\F\}$. 
In the descriptions of the algebras in characteristic 3 in (4/3), 
we use a subspace $\K_{\beta}$ defined as 
$\K_{\beta}=\{\beta\delta^3+\delta\mid \delta\in\F\}$. 

We note that if $\left<x_1,\ldots,x_k\right>$ is a basis for a Lie algebra $L$, 
then any $[p]$-map on $L$ is determined by the images 
$\res x_1,\ldots,\res x_k$.

\begin{theorem}
Suppose that $L$ is a nilpotent Lie algebra of dimension at most 
$4$ over a perfect field $\F$. Then the equivalence classes of 
the $[p]$-maps on $L$ are as follows.
\begin{enumerate}
\item[(1/1)] If $\dim L=1$ and $L=\left<x_1\right>$:
\begin{enumerate}
\item $\res x_1=0$.
\end{enumerate}
\item[(1/1)] If $L=\left<x_1,x_2\right>$:
\begin{enumerate}
\item $\res x_1=\res x_2=0$;
\item $\res x_1=x_2$, $\res x_2=0$.
\end{enumerate}
\item[(3/1)] If $L=\left<x_1,x_2,x_3\right>$:
\begin{enumerate}
\item $\res x_1=\res x_2=\res x_3=0$;
\item $\res x_1=x_2, \res x_2=\res x_3=0$;
\item $\res x_1=x_2$, $\res x_2=x_3$, $\res x_3=0$.
\end{enumerate}
\item[(3/2)] Suppose that 
$L=\left<x_1,x_2,x_3\mid [x_1,x_2]=x_3\right>$.
If $\char\F\geq 3$:
\begin{enumerate}
\item $\res x_1=\res x_2\res x_3=0$;
\item $\res x_1=x_3$, $\res x_2=\res x_3=0$.
\end{enumerate}
If $\char\F=2$:
\begin{enumerate}
\item $\ress x_1=x_3$, $\ress x_2=\xi x_3$, $\ress x_3=0$.
\end{enumerate}
The parameters $\xi_1$ and $\xi_2$ result in equivalent $[p]$-maps
if and only if $\xi_1+\xi_2\in \K$. 
\item[(4/1)] If $L=\left<x_1,x_2,x_3,x_4\right>$:
\begin{enumerate}
\item $\res x_1=\res x_2=\res x_3=\res x_4=0$;
\item $\res x_1= x_2$, $\res x_1=\res x_2=\res x_3=0$;
\item $\res x_1=x_2$, $\res x_3=x_4$, $\res x_2=\res x_4=0$;
\item $\res x_1=x_2$, $\res x_2=x_3$, $\res x_3=\res x_4=0$;
\item $\res x_1=x_2$, $\res x_2=x_3$, $\res x_3=x_4$, $\res x_4=0$.
\end{enumerate}
\item[(4/2)] Suppose that 
$L=\left<x_1,x_2,x_3,x_4\mid [x_1,x_2]=x_3\right>$. If
$\char\F \geq 3$:
\begin{enumerate}
\item $\res x_1=\res x_2=\res x_3=\res x_4=0$;
\item $\res x_1=x_3$, $\res x_2=\res x_3=\res x_4=0$;
\item $\res x_1=x_4$, $\res x_2=\res x_3=\res x_4=0$;
\item $\res x_1=x_3$, $\res x_2=x_4$, $\res x_3=\res x_4=0$;
\item $\res x_3=x_4$, $\res x_1=\res x_2=\res x_4=0$;
\item $\res x_3=x_4$, $\res x_2=x_3$, $\res x_1=\res x_4=0$;
\item $\res x_4=x_3$, $\res x_1=\res x_2=\res x_3=0$;
\item $\res x_4=x_3$, $\res x_2=x_4$, $\res x_1=\res x_3=0$.
\end{enumerate}
If $\char\F=2$:
\begin{enumerate}
\item $\ress x_1=x_3$, $\ress x_2=\xi x_3$, $\ress x_3=\ress x_4=0$;
\item $\ress x_1=x_4$, $\ress x_2=\ress x_3=\ress x_4=0$;
\item $\ress x_1=x_3$, $\ress x_2=x_4$;
\item $\ress x_3=x_4$, $\ress x_1=x_3$, $\ress x_2=\xi x_3$;
\item $\ress x_4=x_3$, $\ress x_1=\ress x_2=\ress x_3=0$;
\item $\ress x_4=x_3$, $\ress x_2=x_4$, $\ress x_1=\ress x_3=0$.
\end{enumerate}
In cases (a) and (d), 
the parameters $\xi_1,\ \xi_2$ represent equivalent $[p]$-maps if 
and only if $\xi_1+\xi_2\in\K$. 
\item[(4/3)] Suppose that $L=\left<x_1,x_2,x_3,x_4\mid [x_1,x_2]=x_3,\ 
[x_1,x_3]=x_4\right>$. If $\char\F\geq 5$:
\begin{enumerate}
\item $\res x_1=\res x_2=\res x_3=\res x_4=0$;
\item $\res x_1=x_4$, $\res x_2=\res x_3=\res x_4=0$;
\item $\res x_2=\xi x_4$, $\res x_1=\res x_3=\res x_4=0$;
\item $\res x_3=x_4$, $\res x_1=\res x_2=\res x_4=0$.
\end{enumerate}
The parameters $\xi_1$ and $\xi_2$ represent isomorphic algebras if and only
if $\xi_1\xi_2^{-1}$ is a square in $\F$.\\
If $\char \F=3$:
\begin{enumerate}
\item $\resss x_1=\resss x_2=\resss x_3=\resss x_4=0$;
\item $\resss x_3= x_4$, $\ress x_1=\resss x_2=\resss x_4=0$;
\item $\resss x_1=\alpha x_4$, $\resss x_2=\beta x_4$, $\resss x_3=\resss x_4=0$.
\end{enumerate}
Where $\alpha\in\F$, $\beta\in\F^*$ and the pairs $(\alpha_1,\beta_1)$, 
$(\alpha_2,\beta_2)$ represent equivalent $[p]$-maps if and only if 
$\beta_2/\beta_1$ is a square in $F$ and 
$\alpha_1\pm\alpha_2\sqrt{\beta_2/\beta_1}\in\K_{\beta_1}$. 
If $\char\F=2$, then $L$ is not restrictable.
\end{enumerate}
\end{theorem}

Through out the paper, we assume that $\F$ is a perfect field  of characteristic $p$ and unless otherwise stated $p\geq 3$.
Thus, the Frobenius automorphism of $\F$ given by   $x\mapsto x^p$ is invertible and we denote the inverse image of $x$ by $x^{1/p}$.

\section{Abelian Lie algebras}

In this section we classify abelian $[p]$-nilpotent restricted Lie 
algebras. In dimension~1, the only $[p]$-nilpotent restricted 
Lie algebra is $\left<x\mid \res x=0\right>$. 
Suppose that $\dim L=2$.
It is not hard to see that there are two possible $[p]$-maps on $L$ that result
in
abelian $[p]$-nilpotent Lie algebras. More precisely,
 there exist linearly independent elements $x_1, x_2\in L$ 
such that either $\res x_1=\res x_2=0$ or $\res x_1=x_2$ and $\res x_2=0$.

Suppose now that $L$ is abelian with $\dim L=3$ and 
let $\{x_1, x_2, x_3\}$ be a basis of $L$. Since $L$ is $[p]$-nilpotent, without loss of generality we assume that $\res x_3=0$.  Let $H=L/\la x_3\ra$. Since $\dim H=2$,  we may assume by the
argument of the first paragraph of the section that either 
$\res x_1, \res x_2\in \la x_3\ra$ or $\res x_1-x_2, \res x_2\in \la x_3\ra$. 

First consider the case $\res x_1, \res x_2\in \la x_3\ra$. Hence $\res x_1=\alpha x_3$ and $\res x_2=\beta x_3$, for some $\alpha, \beta\in \F$. If $\alpha\neq 0$ then we rescale  $x_1$ so that 
$\res x_1= x_3$. Similarly, if  $\beta\neq 0$ then we rescale  $x_2$ so that 
$\res x_2= x_3$. If $\res x_1= x_3$ and $\res x_2= x_3$ then $(-x_1+x_2)^{[p]}= 0$. In this case, we replace $x_2$ with $-x_1+x_2$
to obtain $\res x_2=0$.  
Thus 
we conclude that, up to isomorphism, the possible $[p]$-maps on $L$ are as follows:
\begin{align*}
&\res x_1=\res x_2=\res x_3=0;\\
&\res x_1=x_3,\  \res x_2=\res x_3=0.
\end{align*}

Now consider the case that $\res x_1-x_2, \res x_2\in \la x_3\ra$. 
Hence $\res x_1=x_2+\alpha x_3$ and $\res x_2=\beta x_3$, for some $\alpha, \beta\in \F$.
We replace $x_2$ with $x_2+\alpha x_3$ to obtain $\res x_1=x_2$.
If  $\beta\neq 0$ then we rescale  $x_3$ so that 
$\res x_2= x_3$. We conclude that, up to isomorphism, possible $[p]$-maps on $L$ are as follows:
\begin{align*}
&\res x_1=x_2, \ \res x_2=\res x_3=0;\\
&\res x_1=x_2,\  \res x_2=x_3, \ \res x_3=0.
\end{align*}

The first algeba is isomorphic to 
one of the  algebras above. Thus, 
up to isomorphism,  3-dimensional abelian  $[p]$-nilpotent restricted Lie algebras are as follows:

\begin{eqnarray*}
L_{3,1}^1&=&\left<x_1, x_2, x_3\mid \res x_1=\res x_2=\res x_3=0\right>;\\
L_{3,1}^2&=&\left<x_1, x_2, x_3\mid \res x_1=x_2,\  \res x_2=\res x_3=0\right>;\\
L_{3,1}^3&=&\left<x_1, x_2, x_3\mid \res x_1=x_2,\  \res x_2=x_3, \ \res x_3=0\right>.
\end{eqnarray*}

Let $x_1, x_2, x_3, x_4$ be a basis of a 
4-dimensional abelian  Lie algebra $L=L_{4,1}$. Since $L$ is $[p]$-nilpotent, without loss of generality we assume that $\res x_4=0$.
We set $H=L/\la x_4\ra$ and perform similar calculations as in the 
lower-dimensional cases
to show that, up to isomorphism, 4-dimensional abelian $[p]$-nilpotent restricted Lie algebras are as follows:

\begin{eqnarray*}
L_{4,1}^1&=&\left<x_1, x_2, x_3, x_4\mid \res x_1=\res x_2=\res x_3=\res x_4=0\right>;\\
L_{4,1}^2&=&\left<x_1, x_2, x_3, x_4\mid \res x_1=x_2,\  \res x_2=\res x_3=\res x_4=0\right>;\\
L_{4,1}^3&=&\left<x_1, x_2, x_3, x_4\mid \res x_1=x_2,\  \res x_3= x_4,\ \res x_2=\res x_4=0\right>;\\
L_{4,1}^4&=&\left<x_1, x_2, x_3, x_4\mid \res x_1=x_2,\  \res x_2=x_3, \ \res x_3=\res x_4=0\right>;\\
L_{4,1}^5&=&\left<x_1, x_2, x_3, x_4\mid \res x_1=x_2,\  \res x_2=x_3, \ \res x_3=x_4, \ \res x_4=0\right>.
\end{eqnarray*}

\section{The Heisenberg Lie algebra}

Consider the Heisenberg Lie algebra
$$
L=\left<x_1,\ x_2,\ x_3\mid [x_1, x_2]=x_3\right>.
$$ 
Since $(\ad x)^p=0$ for all $p\geq 2$, 
the image of a $[p]$-map on $L$ is in $\centr L$.
Let $\varphi:L\rightarrow \centr L$ be a $[p]$-map from 
$L$ to $\centr L=\la x_3\ra$. 
Then $\varphi$ can be described by the images of $x_1$, $x_2$, and $x_3$. The 
fact that $(L,\varphi)$ is $[p]$-nilpotent implies that $x_3\varphi=0$. Hence 
$\varphi$ is described by a vector $(\alpha,\beta)$ where
$$
x_1\varphi=\alpha x_3\quad\mbox{and}\quad x_2\varphi=\beta x_3.
$$
The automorphism group of $L$, acting on row vectors 
with respect to the given basis,
consists of the invertible $3\times 3$-matrices of the form
$$
\begin{pmatrix}
a_{11} & a_{12} & a_{13}\\
a_{21} & a_{22} & a_{23}\\
0 & 0 & d
\end{pmatrix}
$$
where $d=a_{11}a_{22}-a_{12}a_{21}$. 
Let $A$ be an automorphism as above and let us compute the 
vector $(\alpha',\beta')$ that describes $A\varphi A^{-1}$.

\subsection{Odd characteristic}
Let us first assume that $p\geq 3$. In this case the $[p]$-map $\varphi$ is
semilinear. Therefore
$$
x_1 A\varphi A^{-1}=(a_{11}x_1+a_{12}x_2+a_{13}x_3)\varphi A^{-1}=
(a_{11}^p\alpha+a_{12}^p\beta) x_3 A^{-1}=(a_{11}^p\alpha+a_{12}^p\beta)d^{-1}x_3.
$$
Hence $\alpha'=(a_{11}^p\alpha+a_{12}^p\beta)d^{-1}$ and we obtain similarly 
that $\beta'=(a_{21}^p\alpha+a_{22}^p\beta)d^{-1}$. 
We claim in this case that $L$
must be isomorphic to one of the 
following algebras.

\begin{eqnarray*}
L_{3,2}^1&=&\left<x_1,\ x_2,\ x_3\mid [x_1, x_2]=x_3,\ \res x_1=\res x_2=\res x_3=0\right>;\\
L_{3,2}^2&=&\left<x_1,\ x_2,\ x_3\mid [x_1, x_2]=x_3,\ \res x_1=x_3,\ \res x_2=\res x_3=0\right>.\\
\end{eqnarray*}
The algebras $L_{3,2}^1$ and $L_{3,2}^2$ are clearly non-isomorphic, 
as $\res{(L_{3,2}^1)}=0$ while $\res{(L_{3,2}^2)}=\left<x_3\right>$. 

Let $\varphi:L\rightarrow \centr L$ be a semilinear transformation.
Then  $x_3\varphi=0, x_1\varphi=\alpha x_3$, and $ x_2\varphi=\beta x_3$, for some $\alpha, \beta\in \F$. If $\alpha=\beta=0$ then we have  $L_{3,2}^1$. Suppose that
this is not the case and assume without loss of generality that
$\alpha\neq 0$.  Note that $(c x_1+x_2)\varphi=(c^p\alpha+\beta)x_3$, for every $c\in \F$. Hence choosing $c_0=-(\beta/\alpha)^{1/p}$ and
$x_2'=c_0x_1+x_2$, we obtain $x_2'\varphi=0$. Then replacing  $x_3$ with $\alpha x_3$ and $x_2$ with $\alpha x_2$ we find that $x_1\varphi=x_3$
and hence we obtain $L_{3,2}^2$.

\subsection{Characteristic 2}\label{HeiChar2}
Suppose that $L$ is the Heisenberg Lie
algebra over a field $\F$ of characteristic 2. In this section we classify the possible $[p]$-nilpotent
$[p]$-maps on  $L$.
First we set 
\begin{equation}\label{kdef}
\K=\left\{\delta+\delta^2\mid\delta\in\F\right\}.
\end{equation}
As the characteristic is 2,
$\K$ is an $\F_2$-subspace of $\F$. In fact,
$\K$ is the image of the $\F_2$-linear map $\delta\mapsto \delta+\delta^2$
whose kernel is $\F_2$. Hence if $\F$ is a finite field, $\K$ has 
co-dimension~1, but this may not be true for other fields. For instance
if $\F$ is algebraically closed, then the polynomial $t^2+t+\alpha$ has a 
root for any $\alpha\in\F$ and so $\K=\F$. The subspace $\K$ is often referred to
as the Artin-Schreier subspace; see for instance~\cite[page~70]{Cohen}.

We claim that $L$ is isomorphic to a Lie algebra of the form
$$
K_{3,2}^1(\xi)=\left<x_1,x_2,x_3\mid [x_1,x_2]=x_3,\ \ress x_1=x_3,\ \ress x_2=\xi x_3,\ \ress x_3=0\right>,
$$
where $\xi\in\F$. 
Further, $K_{3,2}^1(\xi_1)\cong K_{3,2}^1(\xi_2)$ 
if and only if $\xi_1+\xi_2\in \K$. 

 Note that the map $\varphi:L\rightarrow L$ defined by $x\mapsto \ress x$ is not
semilinear. However, since $0=(\ad x_1)^2=\ad x_1^{[2]}$ and $0=(\ad x_2)^2=\ad
x_2^{[2]}$, we have that $\varphi:L\rightarrow \centr L$. Suppose that 
$\varphi$ is the $[p]$-map represented by the vector $(\alpha,\beta)$ as
in the odd characteristic case. Then computing $A\varphi  A^{-1}$ 
for an automorphism $A=(a_{i,j})$ we obtain
that $A\varphi  A^{-1}$ is represented by the vector $(\alpha',\beta')$
where
\begin{eqnarray*}
\alpha'&=&d^{-1}(\alpha a_{11}^2+\beta a_{12}^2+a_{11}a_{12})\\
\beta'&=&d^{-1}(\alpha a_{21}^2+\beta a_{22}^2+a_{21}a_{22})
\end{eqnarray*}
where $d=a_{11}a_{22}+a_{12}a_{21}$. 
First we show that every Lie algebra is isomorphic to $L_{3,2}^1(\xi)$ with
some $\xi\in\F$. If $\alpha=\beta=0$, then replace $x_1$ with $x_1+x_2$
to obtain $\ress x_1=x_3$. If $\alpha\neq 0$ then replace $x_2$ and $x_3$ 
with $\alpha x_2$ and $\alpha x_3$, respectively, to obtain $\ress{x_1}=x_3$. 
If $\alpha=0$, but $\beta\neq 0$, then swap $x_1$ and $x_2$ and repeat
the steps in the previous sentence. Hence 
$L$ is isomorphic to $L_{3,2}^1(\xi)$ with some $\xi\in\F$, as claimed.

Let us now show the claim concerning the isomorphisms between the algebras
$K_{3,2}^1(\xi)$. Suppose, for $i=1,\ 2$, that $\varphi_i$ 
is represented by the vector $(1,\xi_i)$.
First if $\xi_1+\xi_2\in \K$, then
choose $\delta\in\F$ such that $\xi_1+\delta^2+\delta=\xi_2$ and consider 
the automorphism $A$ with $a_{11}=1$, $a_{12}=0$, $a_{21}=\delta$ and $a_{22}=1$. 
Then $A\varphi_1 A^{-1}=\varphi_2$. Therefore $L_{3,2}^1(\xi_1)\cong 
L_{3,2}^1(\xi_2)$. 

Conversely, suppose that $K_{3,2}^1(\xi_1)\cong K_{3,2}^1(\xi_2)$, that is, 
there is some $A\in \aut L$ such that $A \varphi_1 A^{-1}=\varphi_2$. 
First we note that diagonal automorphisms of the form $\diag{a,a,a^2}$ stabilize
$\varphi_1$ for all $a\in\F^*$. Suppose that $a_{22}=0$. Let $d=a_{11}a_{22}+a_{12}a_{21}=a_{12}a_{21}$.
Then swapping $A$ with $\diag{d^{-1/2},d^{-1/2},d^{-1}}A$, we
may, and will, assume that $a_{12}a_{21}=d=1$ and hence $a_{21}=a_{12}^{-1}$.  
Thus $A\varphi_1A^{-1}=\varphi_2$ implies that
\begin{eqnarray*}
a_{11}^2+\xi_1 a_{12}^2+a_{11}a_{12}&=&1;\\
a_{12}^{-2}&=&\xi_2.
\end{eqnarray*}
Combining these equations, we find $\xi_1+\xi_2=(a_{11}/a_{12})^2+a_{11}/a_{12}\in \K$,
as required.

Assume now $a_{22}\neq 0$ and  $d=a_{11}a_{22}+a_{12}a_{21}$.
We may suppose, as above, that $d=a_{11}a_{22}+a_{12}a_{21}=1$. Thus 
we obtain
\begin{eqnarray}
\label{eq1}a_{11}^2+\xi_1a_{12}^2+a_{11}a_{12}&=&1;\\
\label{eq2}a_{11}a_{22}+a_{12}a_{21}&=&1;\\
\label{eq3}a_{21}^2+\xi_1a_{22}^2+a_{21}a_{22}&=&\xi_2.
\end{eqnarray}
Equation~\eqref{eq2} implies that $a_{11}=(1+a_{12}a_{21})/a_{22}$ which we 
substitute into equation~\eqref{eq1} to obtain that
$$
1+a_{12}^2a_{21}^2+\xi_1a_{12}^2a_{22}^2+a_{22}a_{12}+a_{22}a_{12}^2a_{21}=a_{22}^2
$$
and hence 
\begin{equation}\label{eq4}
a_{22}^2=1+a_{12}^2(a_{21}^2+\xi_1a_{22}^2+a_{22}a_{21})+a_{22}a_{12}=1+a_{12}^2\xi_2+a_{22}a_{12}.
\end{equation}
If $a_{12}=0$ then this implies that $a_{22}=1$, and then equation~\eqref{eq3} 
gives that $\xi_1+\xi_2\in \K$ as required. If $a_{12}\neq 0$, then 
equation~\eqref{eq4} gives that $\xi_2+a_{12}^{-2}\in \K$. On the other hand, 
in this case, equation~\eqref{eq1} gives that $\xi_1+a_{12}^{-2}\in \K$ and hence
$\xi_1+\xi_2\in \K$, as claimed.

\section{Restriction maps on $L_{4,2}$}

Consider  the Lie algebra
$$
L=L_{4,2}=\left<x_1, x_2,x_3,x_4\mid [x_1, x_2]=x_3\right>.
$$
 The automorphism group of $L$ consists of the set of invertible matrices of the
form
$$
\begin{pmatrix}
a_{11} & a_{12} & a_{13} & a_{14}\\
a_{21} & a_{22} & a_{23} & a_{24}\\
0 & 0 & d & 0\\
0 & 0 & a_{43} & a_{44}\\
\end{pmatrix},
$$
where $d=a_{11}a_{22}-a_{12}a_{21}$. 
Note that $\centr L=\left<x_3,x_4\right>$ has dimension 2.  Since $L$ 
is $[p]$-nilpotent, we have 
$$
\centr L^{[p]^2} <\centr L^{[p]} < \centr L.
$$
In particular, $\centr L^{[p]^2} =0$.
Since the nilpotency class of $L$ is two,  
a restricted Lie algebra structure on 
$L$ is given by a semilinear
transformation  from $L$ to $\centr L$ whenever $p\geq 3$.

First we sort out the possible restrictions of a $[p]$-map on $\centr L$, for every $p\geq 2$. 
Suppose that $\res x_3\neq 0$. Then $\res x_3=\alpha x_3+\beta x_4$. If $\beta=0$, 
then $\res x_3=\alpha x_3$, and so $x_3^{[p]^n}\neq 0$ which contradicts 
the assumption that $L$ is  $[p]$-nilpotent. Hence $\beta\neq 0$ and 
we may replace $x_4$ with $x_4'=\alpha x_3+\beta x_4$. This
replacement is an automorphism of $L$  and then we get $\res x_3=x_4$. Now $\res x_4=\gamma x_3+\delta x_4$, which gives
that 
$$
0=x_4^{[p]^2}=\res{(\gamma x_3+\delta x_4)}=
(\gamma^p+\delta^p\delta)x_4+\delta^p\gamma x_3.
$$
Hence, $\delta=\gamma=0$ and so $\res x_4=0$. 

If $\res x_3=0$ and $\res x_4\neq 0$, then $\res x_4=\alpha x_3+\beta x_4$. We have:
$$
0=x_4^{[p]^2}=\res{(\alpha x_3+\beta x_4)}=\beta^p\res x_4
$$
which shows that $\beta=0$ and hence $\res x_4=\alpha x_3$. Replacing $x_4$ with
$(1/\alpha)^{1/p}x_4$ is a Lie algebra automorphism and  results in
$\res x_4=x_3$. 

Thus, the map $\varphi$ on $\centr L$ can have three different 
forms: either $\res x_3=\res x_4=0$; or $\res x_3=x_4$ and $\res x_4=0$;
or $\res x_3=0$ and $\res x_4=x_3$. We will consider these three 
cases separately.

\subsection{}\label{case1}
Suppose first that $\res x_3=\res x_4=0$ and let $\varphi$ be a semilinear 
map that extends this $[p]$-map to the whole $L$. Let 
 $x_1\varphi=\alpha_1 x_3+\beta_1 x_4$ and
$x_2\varphi=\alpha_2 x_3+\beta_2 x_4$. Let $A\in\aut L$ and 
consider the semilinear transformation $\varphi'=A\varphi A^{-1}$ given by
$x_1\varphi'=\alpha_1' x_3+\beta_1' x_4$ and
$x_2\varphi'=\alpha_2'x_3+\beta_2' x_4$. Let us compute the coefficients
$\alpha_1',\ \alpha_2',\ \beta_1',\ \beta_2'$. First, we note that 
$$
x_3A^{-1}=d^{-1}x_3\quad\mbox{and}\quad
x_4A^{-1}=-a_{43}/(da_{44})x_3+a_{44}^{-1}x_4.
$$
Now
\begin{align*}
x_1A\varphi A^{-1}&=(a_{11}x_1+a_{12}x_2+a_{13}x_3+a_{14}x_4)\varphi A^{-1}\\
&=((\alpha_1a_{11}^p+\alpha_2a_{12}^p)x_3+(\beta_1a_{11}^p+\beta_2a_{12}^p)x_4)A^{-1}\\
&=
\bigg(d^{-1}(\alpha_1a_{11}^p+\alpha_2a_{12}^p)-\frac{a_{43}}{da_{44}}(\beta_1a_{11}^p+\beta_2a_{12}^p)\bigg)x_3+
a_{44}^{-1}(\beta_1a_{11}^p+\beta_2a_{12}^p)x_4.
\end{align*}
Repeating the calculation for  $x_2A\varphi A^{-1}$, we obtain that
\begin{eqnarray*}
\alpha_1'&=&d^{-1}(\alpha_1a_{11}^p+\alpha_2a_{12}^p)-\frac{a_{43}}{da_{44}}(\beta_1a_{11}^p+\beta_2a_{12}^p)\\
\alpha_2'&=&d^{-1}(\alpha_1a_{21}^p+\alpha_2a_{22}^p)-\frac{a_{43}}{da_{44}}(\beta_1a_{21}^p+\beta_2a_{22}^p)\\
\beta_1'&=&a_{44}^{-1}(\beta_1a_{11}^p+\beta_2a_{12}^p)\\
\beta_2'&=&a_{44}^{-1}(\beta_1a_{21}^p+\beta_2a_{22}^p).
\end{eqnarray*}
Thus
\begin{align}\label{aut}
(\alpha_1',\alpha_2',\beta_1',\beta_2')=
(\alpha_1,\alpha_2,\beta_1,\beta_2)
\begin{pmatrix}
d^{-1}a_{11}^p & d^{-1}a_{21}^p & 0 & 0\\
d^{-1}a_{12}^p & d^{-1}a_{22}^p & 0 & 0\\
-\frac{a_{43}}{da_{44}}a_{11}^p & -\frac{a_{43}}{da_{44}}a_{21}^p & a_{44}^{-1}a_{11}^p & a_{44}^{-1}a_{21}^p\\
-\frac{a_{43}}{da_{44}}a_{12}^p & -\frac{a_{43}}{da_{44}}a_{22}^p & a_{44}^{-1}a_{12}^p & a_{44}^{-1}a_{22}^p
\end{pmatrix}.
\end{align}
Hence, the action of  the automorphism in \eqref{aut}  on the vector space $\F^4$ can be described by the
tensor product
\begin{align}
\begin{pmatrix}\label{tensor}
d^{-1} & 0 \\
-\frac{a_{43}}{da_{44}} & a_{44}^{-1}
\end{pmatrix}
\otimes\begin{pmatrix}
a_{11}^p & a_{21}^p\\
a_{12}^p & a_{22}^p
\end{pmatrix}
\end{align}
acting on the tensor product  $V_1\otimes V_2$, where $V_1=V_2=\F^2$. Let us 
calculate the orbits under this action.   We denote the group of matrices of the form \eqref{tensor} by $H$. Let $e_1, e_2$ and $f_1, f_2$ be the standard bases of $V_1$ and $V_2$, respectively.

Let $W=V_1\otimes V_2$ and let $v\in W\setminus 0$. First suppose that
$v=v_1\otimes v_2$ with $v_1\in V_1\setminus 0$ and $v_2\in V_2\setminus 0$. 
Since $\gl 2\F$ is transitive on the non-zero vectors of $V_2$, there exists $g_2\in \gl 2\F$ such that $v_2g_2=(1,0)$. Choose
$g_1$ such that $g_1\otimes g_2\in H$. Then 
$$
(v_1\otimes v_2)(g_1\otimes g_2)=v_1'\otimes (1,0).
$$
Let us now consider vectors of the form $(\alpha,\beta)\otimes (1,0)$. If
$\beta=0$ we have
$$
\left((\alpha,0)\otimes (1,0)\right)\left(\begin{pmatrix}
\alpha^{-1} & 0 \\
0 & 1
\end{pmatrix}
\otimes\begin{pmatrix}
1 & 0\\
0 & \alpha^p
\end{pmatrix}\right)
=(1,0)\otimes(1,0).
$$
If $\beta\neq 0$ then 
$$
\left((\alpha,\beta)\otimes (1,0)\right)\left(\begin{pmatrix}
1 & 0 \\
-\alpha/\beta & \beta^{-1}
\end{pmatrix}
\otimes\begin{pmatrix}
1 & 0\\
0 & 1
\end{pmatrix}\right)
=(0,1)\otimes(1,0).
$$
We deduce  that the group $H$ has three orbits on the set of 
pure tensors with orbit representatives $0,(1,0)\otimes (1,0)$ and 
$(0,1)\otimes (1,0)$. 

Let us compute the orbits of $H$ on the set of elements 
that are not pure tensors. Such an element is of the form 
$e_1\otimes v_1+e_2\otimes v_2$ with $v_1$ and $v_2$ linearly independent.
Note that there exists a $g^p\in \gl 2\F$ that maps $v_1\mapsto f_1$
and $v_2\mapsto f_2$. Let $d=\det g$. Then 
$$
(e_1\otimes v_1+e_2\otimes v_2)
\left(\begin{pmatrix}
d & 0 \\
0 & 1
\end{pmatrix}
\otimes g\right)=d e_1\otimes f_1+e_2\otimes f_2.
$$
Hence every orbit contains an element of the form 
$\alpha e_1\otimes f_1+e_2\otimes f_2$ with $\alpha\in\F^*$. Now 
$$
(\alpha e_1\otimes f_1+e_2\otimes f_2)\left(
\begin{pmatrix}
\alpha^{-1} & 0 \\
0 & \alpha^{-p}
\end{pmatrix}
\otimes\begin{pmatrix}
1 & 0\\
0 & \alpha^p
\end{pmatrix}\right)
=e_1\otimes f_1+e_2\otimes f_2.
$$

Hence, $H$ has 4 orbits  and 
the representatives of these orbits are $0$, $(1,0)\otimes (1,0)$, 
$(0,1)\otimes (1,0)$, and $(1,0)\otimes (1,0)+(0,1)\otimes (0,1)$. 
The corresponding restricted Lie algebras are
\begin{eqnarray*}
L_{4,2}^1&=&\left<x_1,x_2,x_3,x_4\mid [x_1,x_2]=x_3\right>;\\
L_{4,2}^2&=&\left<x_1,x_2,x_3,x_4\mid [x_1,x_2]=x_3,\ \res x_1=x_3\right>;\\
L_{4,2}^3&=&\left<x_1,x_2,x_3,x_4\mid [x_1,x_2]=x_3,\ \res x_1=x_4\right>;\\
L_{4,2}^4&=&\left<x_1,x_2,x_3,x_4\mid [x_1,x_2]=x_3,\ \res x_1=x_3,\ \res x_2=x_4\right>.
\end{eqnarray*}

\subsubsection{Characteristic $2$}\label{l422ch2}
Let us now assume that $\ress x_3=\ress x_4=0$ and $\char\F=2$.
In this case we will show that 
$L$ is isomorphic to one of the following Lie algebras:
\begin{eqnarray*}
K_{4,2}^1(\xi)&=&\left<x_1,x_2,x_3,x_4\mid [x_1,x_2]=x_3,\ \ress x_1=x_3,\ 
\ress x_2=\xi x_3\right>;\\
K_{4,2}^2&=&\left<x_1,x_2,x_3,x_4\mid [x_1,x_2]=x_3,\ \ress x_1=x_4\right>;\\
K_{4,2}^3&=&\left<x_1,x_2,x_3,x_4\mid [x_1,x_2]=x_3,\ \ress x_1=x_3,\ 
\ress x_2=x_4\right>.
\end{eqnarray*}
We claim further that $K_{4,2}^1(\xi_1)\cong K_{4,2}^1(\xi_2)$ if and only if
$\xi_1+\xi_2\in \K$ where $\K$ is the Artin-Schreier subspace defined
in~\eqref{kdef}.

First assume that $\ress L\leq \left<x_3\right>$. In this case 
$L=L_1\oplus\left<x_4\right>$ where $L_1=\left<x_1,x_2,x_3\right>$ and 
the 
direct sum is interpreted as a direct sum of restricted 
ideals. Hence $L$ is isomorphic
to an algebra of the form $K_{3,2}^1(\xi)\oplus\F$ where $K_{3,2}^1(\xi)$ is an 
algebra defined in Section~\ref{HeiChar2}. Thus $L\cong K_{4,2}^1(\xi)$ with
some $\xi\in\F$. 

Hence we may assume that $\ress L\not\leq\left<x_3\right>$. 
Assume that $\ress x_1=\alpha_1 x_3+\beta_1 x_4$ and that 
$\ress x_2=\alpha_2 x_3+\beta_2 x_4$. 
Either by swapping $x_1$ and $x_2$ or replacing $x_1$ with $x_1+x_2$, 
we may assume without loss of generality that $\ress x_1\neq 0$. 
Assume first that $\ress x_2=0$. 
As $\ress L\not\leq\left<x_3\right>$,  we must have $\beta_1\neq 0$.
Then replace $x_4$ with $\alpha_1 x_3+\beta_1 x_4$ to obtain that 
$L\cong K_{4,2}^2$. 

Suppose now that $\ress x_1\neq 0$ and $\ress x_2\neq 0$. If $\beta_2\neq 0$, then
we replace $x_4$ with $\alpha_2 x_3+\beta_2 x_4$ and obtain that 
$\ress x_2=x_4$. Now $(x_1+\beta_1^{1/2}x_2)^{[2]}=\alpha_1 x_3$. 
Hence replacing $x_1$ with $x_1+\beta_1^{1/2}x_2$ we may assume that 
$\ress x_1=\alpha_1 x_3$ with some $\alpha_1\in\F^*$. 
Now replace $x_2$ by $\alpha_1 x_2$, $x_3$ by 
$\alpha_1x_3$ and $x_4$ by $\alpha_1^2x_4$ to obtain that $\ress x_1=x_3$ and
$\ress x_2=x_4$ and hence $L\cong K_{4,2}^3$. If $\beta_1\neq 0$ then 
swapping $x_1$ and $x_2$ allows us to repeat this argument.

These Lie algebras are pairwise non-isomorphic,  as 
$K_{4,2}^1(\xi)$ 
are the only ones with $\ress L=L'$, $K_{4,2}^2$ is the only one with
$\ress L\cap L'=0$, while $K_{4,2}^3$ is the only one
with $L'<\ress L$. The claim concerning the isomorphisms among the algebras
$K_{4,2}^1(\xi)$ follows from the fact that if $L$ is such an algebra and $I$
is a one-dimensional ideal, then $I\leq \centr L=\left<x_3,x_4\right>$, and so
$L/I$ is either abelian (if and only if $I=\left<x_3\right>$) or is isomorphic
to $K_{3,2}^1(\xi)$ in Section~\ref{HeiChar2}.

\subsection{}\label{case2}
Let us  now consider the case when $\res x_3=x_4$ and $\res x_4=0$.
First we note that if $A$ is an automorphism of $L$ then $A$ preserves 
$\varphi|_{\centr L}$ if and only if $a_{43}=0$ and $d^p=a_{44}$.
Let  $x_1\varphi=\alpha_1 x_3+\beta_1 x_4$ and
$x_2\varphi=\alpha_2x_3+\beta_2 x_4$.
Now consider the semilinear transformation $\varphi'=A\varphi A^{-1}$ given by
$x_1\varphi'=\alpha_1' x_3+\beta_1' x_4$ and
$x_2\varphi'=\alpha_2'x_3+\beta_2' x_4$.
\subsubsection{Odd characteristic} 
Similar  calculations as in Section \ref{case1} show that
\begin{eqnarray*}
\alpha_1'&=&d^{-1}(\alpha_1a_{11}^p+\alpha_2a_{12}^p)\\
\alpha_2'&=&d^{-1}(\alpha_1a_{21}^p+\alpha_2a_{22}^p)\\
\beta_1'&=&a_{44}^{-1}(\beta_1a_{11}^p+\beta_2a_{12}^p+a_{13}^p)\\
\beta_2'&=&a_{44}^{-1}(\beta_1a_{21}^p+\beta_2a_{22}^p+a_{23}^p)
\end{eqnarray*}
Let us write $$
(\alpha_1,\alpha_2,\beta_1,\beta_2)(A\varrho )=(\alpha_1',\alpha_2',\beta_1',\beta_2')
$$
Thus, in matrix form we have:
\begin{equation*}
(\alpha_1,\alpha_2,\beta_1,\beta_2)(A\varrho )=(\alpha_1,\alpha_2,\beta_1,\beta_2)
\begin{pmatrix}
d^{-1}a_{11}^p &d^{-1}a_{21}^p &0& 0\\
d^{-1}a_{12}^p&d^{-1}a_{22}^p & 0 & 0\\
0& 0&a_{44}^{-1}a_{11}^p& a_{44}^{-1}a_{21}^p \\
0& 0&a_{44}^{-1}a_{12}^p&a_{44}^{-1}a_{22}^p \\
\end{pmatrix}+
\begin{pmatrix}
0\\
0\\
a_{44}^{-1}a_{13}^p\\
a_{44}^{-1}a_{23}^p
\end{pmatrix}.
\end{equation*}

We claim that there are two orbits with representatives $(0,0,0,0)$ and 
$(0,1,0,0)$.
First notice that 
$$
(\alpha_1,\alpha_2,\beta_1,\beta_2)
\begin{pmatrix}
1 & 0 & a_{13} & 0\\
0 & 1 & a_{23} & 0\\
0 & 0 & 1 & 0\\
0 & 0 & 0 & 1
\end{pmatrix}\varrho
=(\alpha_1,\alpha_2,\beta_1+a_{13}^p,\beta_2+a_{23}^p).
$$
Since $a_{13}$ and $a_{23}$ can be chosen freely, 
$(\alpha_1, \alpha_2,\beta_1,\beta_2)$ and $(\alpha_1,\alpha_2,\beta_3,\beta_4)$
are in the same orbit for all $\beta_1,\beta_2,\beta_3,\beta_4\in\F$.
Let $(\alpha_1,\alpha_2,\beta_1,\beta_2)\in\F^4$. If $\alpha_1=\alpha_2=0$, 
then the set of such elements would form a single orbit with orbit 
representative $(0,0,0,0)$. 
Supose now that 
$(\alpha_1,\alpha_2)\neq (0,0)$. We may assume without loss of generality 
that $\beta_1=\beta_2=0$. 
If $\alpha_2\neq 0$ then
$$
(\alpha_1,\alpha_2, 0,0)\begin{pmatrix}
\alpha_2^{1/p} & -\alpha_1^{1/p} & 0 & 0\\
0 & \alpha_2^{-1/p} & 0 & 0\\
0 & 0 & 1 & 0\\
0 & 0 & 0 & 1
\end{pmatrix}\varrho=(0,1,0,0).
$$
If $\alpha_2=0$ then $\alpha_1\neq 0$ and we obtain that
$$
(\alpha_1,0,0,0)\begin{pmatrix}
0 & -\alpha_1^{1/p} & 0 & 0\\
\alpha_1^{-1/p} & 0 & 0 & 0\\
0 & 0 & 1 & 0\\
0 & 0 & 0 & 1
\end{pmatrix}\varrho=(0,1,0,0).
$$

Hence, there are two orbits with  representatives $(0,0,0,0)$ and 
$(0,1,0,0)$ as claimed. The corresponding Lie algebras are
\begin{eqnarray*}
L_{4,2}^5&=&\left<x_1,x_2,x_3,x_4 \mid [x_1,x_2]=x_3,\ \res x_3=x_4 \right>;\\
L_{4,2}^6&=&\left<x_1,x_2,x_3,x_4 \mid [x_1,x_2]=x_3,\  \res x_2=x_3,\ \res x_3=x_4 \right>.
\end{eqnarray*}

\subsubsection{Characteristic $2$}
Suppose now that $\char\F=2$. 
Suppose that $\K$ and $\xi$ are as in Section~\ref{HeiChar2}. 
We claim that $L$ is isomorphic to 
\begin{eqnarray*}
K_{4,2}^4(\xi)&=&\left<x_1,x_2,x_3,x_4 \mid [x_1,x_2]=x_3,\ \ress x_3=x_4, 
\ress x_1= x_3,\ \ress x_2=\xi x_3 \right>,
\end{eqnarray*}
where $\xi\in\F$. Further, $K_{4,2}^4(\xi_1)\cong K_{4,2}^4(\xi_2)$ if 
and only if $\xi_1+\xi_2\in\K$ as in Section~\ref{HeiChar2}.
If a $[2]$-map is represented by a vector $(\alpha_1,\alpha_2,\beta_1,\beta_2)$
as above and $A\in\aut L$, 
then $A\varphi A^{-1}$ is represented by 
$(\alpha_1',\alpha_2',\beta_1',\beta_2')$ where
\begin{eqnarray*}
\alpha_1'&=&d^{-1}(\alpha_1a_{11}^p+\alpha_2a_{12}^p+a_{11}a_{12})\\
\alpha_2'&=&d^{-1}(\alpha_1a_{21}^p+\alpha_2a_{22}^p+a_{21}a_{22})\\
\beta_1'&=&a_{44}^{-1}(\beta_1a_{11}^p+\beta_2a_{12}^p+a_{13}^p)\\
\beta_2'&=&a_{44}^{-1}(\beta_1a_{21}^p+\beta_2a_{22}^p+a_{23}^p).
\end{eqnarray*}
Note that the quotient $L/\left<x_4\right>$ is isomorphic to 
the Heisenberg Lie algebra and so by Section~\ref{HeiChar2}, we may assume that
$(\alpha_1,\alpha_2,\beta_1,\beta_2)$ is of the form
$(1,\xi,\beta_1,\beta_2)$. Then we use similar  arguments
as in the case of $\char\F\geq 3$ to show that $(1,\xi,\beta_1,\beta_2)$ and $(1,\xi,0,0)$ are 
in  the same orbit.
This shows that our algebra is isomorphic to $K_{4,2}^4(\xi)$ with 
some $\xi\in \F$. To prove the claim concerning the
isomorphisms among these algebras, notice that $\left<x_4\right>$ is the
unique restricted ideal of $K_{4,2}^4(\xi)$ with dimension 1, and
$K_{4,2}^4(\xi)/\left<x_4\right>\cong K_{3,2}^1(\xi)$. Thus 
$K_{4,2}^4(\xi_1)\cong K_{4,2}^4(\xi_2)$ if and only if 
$K_{3,2}^1(\xi_1)\cong K_{3,2}^1(\xi_2)$ if and only if $\xi_1+\xi_2\in\K$.

\subsection{}
Finally, we consider the case where $\res x_3=0$ and $\res x_4=x_3$.
Let $\varphi$ be a semilinear map
that extends this $[p]$-map to the whole $L$. 
Note that if $A$ is an automorphism of $L$ then $A$ preserves 
$\varphi|_{\centr L}$ if and only if $d=a_{44}^p$.
Let  $x_1\varphi=\alpha_1 x_3+\beta_1 x_4$ and
$x_2\varphi=\alpha_2x_3+\beta_2 x_4$.

\subsubsection{Odd characteristic} 
Let $A\in\aut L$ and 
consider the semilinear transformation $\varphi'=A\varphi A^{-1}$ given by
$x_1\varphi'=\alpha_1' x_3+\beta_1' x_4$ and
$x_2\varphi'=\alpha_2'x_3+\beta_2' x_4$.
Then
\begin{eqnarray*}
\alpha_1'&=&d^{-1}(\alpha_1a_{11}^p+\alpha_2a_{12}^p+a_{14}^p)-a_{43}/(da_{44})(\beta_1a_{11}^p+\beta_2a_{12}^p)\\
\alpha_2'&=&d^{-1}(\alpha_1a_{21}^p+\alpha_2a_{22}^p+a_{24}^p)-a_{43}/(da_{44})(\beta_1a_{21}^p+\beta_2a_{22}^p)\\
\beta_1'&=&a_{44}^{-1}(\beta_1a_{11}^p+\beta_2a_{12}^p)\\
\beta_2'&=&a_{44}^{-1}(\beta_1a_{21}^p+\beta_2a_{22}^p).
\end{eqnarray*}

As usual, we denote this action of $\aut L$ by $\varrho$. 
As in Section \ref{case2}, the vectors $(\alpha_1,\beta_1,\alpha_2,\beta_2)$ and
$(\alpha_3,\beta_1,\alpha_4,\beta_2)$ are 
in the same orbit for all $\alpha_i,\ \beta_i$. Hence elements 
of the form $(\alpha_1,0,\alpha_2,0)$ form a single 
orbit with representative $(0,0,0,0)$. 
Suppose now that $(\beta_1,\beta_2)\neq (0,0)$. Consider the vector
$(\alpha_1,\beta_1,\alpha_2,\beta_2)$. We may assume without loss of generality
that $\alpha_1=\alpha_2=0$. Then if $\beta_2\neq 0$ then
$$
(0,\beta_1,0,\beta_2)\begin{pmatrix}
\beta_2^{1/p} & -\beta_1^{1/p} & 0 & 0\\
0 & \beta_2^{-1/p} & 0 & 0\\
0 & 0 & 1 & 0\\
0 & 0 & 0 & 1
\end{pmatrix}\varrho=(0,0,0,1).
$$
If $\beta_2=0$ then
$$
(0,\beta_1,0,0)
\begin{pmatrix}
0 & -\beta_1^{1/p} & 0 & 0\\
\beta_1^{-1/p} & 0 & 0 & 0\\
0 & 0 & 1 & 0\\
0 & 0 & 0 & 1
\end{pmatrix}\varrho=(0,0,0,1).
$$

Hence, there are two orbits with representatives $(0,0,0,0)$ and 
$(0,0,1,0)$, as claimed. The corresponding Lie algebras are
\begin{eqnarray*}
L_{4,2}^7&=&\left<x_1,x_2,x_3,x_4 \mid [x_1,x_2]=x_3,\ \res x_4=x_3 \right>;\\
L_{4,2}^8&=&\left<x_1,x_2,x_3,x_4 \mid [x_1,x_2]=x_3,\ \res x_4=x_3,\ \res x_2=x_4 \right>.
\end{eqnarray*}

\subsubsection{Characteristic 2}
Then 
$A\varphi A^{-1}$ is represented by $(\alpha_1',\beta_1',\alpha_2',\beta_2')$ 
where
\begin{eqnarray*}
\alpha_1'&=&d^{-1}(\alpha_1a_{11}^2+\alpha_2a_{12}^2+a_{14}^2+a_{11}a_{12})-a_{43}/(da_{44})(\beta_1a_{11}^2+\beta_2a_{12}^2)\\
\alpha_2'&=&d^{-1}(\alpha_1a_{21}^2+\alpha_2a_{22}^2+a_{24}^2+a_{21}a_{22})-a_{43}/(da_{44})(\beta_1a_{21}^2+\beta_2a_{22}^2)\\
\beta_1'&=&a_{44}^{-1}(\beta_1a_{11}^2+\beta_2a_{12}^2)\\
\beta_2'&=&a_{44}^{-1}(\beta_1a_{21}^2+\beta_2a_{22}^2).
\end{eqnarray*}
We claim that $L$ is isomorphic to one of the following algebras:
\begin{eqnarray*}
K_{4,2}^5&=&\left<x_1,x_2,x_3,x_4 \mid [x_1,x_2]=x_3,\ \ress x_4=x_3 \right>;\\
K_{4,2}^{6}&=&\left<x_1,x_2,x_3,x_4 \mid [x_1,x_2]=x_3,\ \ress x_4=x_3,\ \ress x_2=x_4 \right>.
\end{eqnarray*}

Similarly as in the case of $\char\F\geq 3$ we obtain that 
every $\aut L$-orbit contains a tuple of the form $(0,\beta_1,0,\beta_2)$. 
Hence we may assume that $\ress x_1=\beta_1 x_4$ and $\ress x_2=\beta_2 x_4$.
If $\beta_1=\beta_2=0$ then $L$ is isomorphic to 
$K_{4,2}^5$. If $\beta_1= 0$ and $\beta_2\neq 0$, then 
apply the diagonal automorphism $\diag{\beta_2^{1/2},\beta_2^{-1/2},1,1}$ 
to obtain $L_{4,2}^{6}$. If $\beta_1=0$ and $\beta_2\neq 0$, then
we can swap $x_1$ and $x_2$ and repeat the argument. Finally 
if $\beta_1\beta_2\neq 0$, then apply the diagonal automorphism 
$\diag{\beta_1^{-3/4}\beta_2^{-1/4},\beta_1^{-1/4}\beta_2^{-3/4},\beta_1^{-1}\beta_2^{-1},\beta_1^{-1/2}\beta_2^{-1/2}}$~~/~~ 
to obtain 
that $\beta_1=\beta_2=1$. Now replace $x_1$ with $x_1+x_2+x_4$ to obtain 
$K_{4,2}^6$. As $\ress{(K_{4,2}^5)}\leq (K_{4,2}^5)'$, 
while this is not the case with $K_{4,2}^6$,
the two algebras are non-isomorphic.

\section{Restriction maps on  $L_{4,3}$}\label{L43}
Consider the Lie algebra
$$
L=L_{4,3}=\left<x_1,\ x_2,\ x_3,\ x_4\mid [x_1,x_2]=x_3,\ [x_1,x_3]=x_4\right>.
$$
The automorphism group of $L$, with respect to the given basis, 
consists of the invertible matrices of the form
$$
\begin{pmatrix}
a_{11} & a_{12} & a_{13} & a_{14}\\
0 & a_{22} & a_{23} & a_{24}\\
0 & 0 & d_1& a_{11}a_{23}\\
0 & 0 & 0 & a_{11}d_1
\end{pmatrix},
$$
where $d_1=a_{11}a_{22} $.
We have  $\centr L=\left<x_4\right>$. Since $L$ is $[p]$-nilpotent, we must have $\res x_4=0$.
First we note that $L_{4,3}$ is not restrictable in characteristic 2. 
For, a
restriction map in characteristic $2$
would have to satisfy $\ad (\ress x)=(\ad x)^2$ for all $x$. 
On the other hand, we have that $(\ad
x_1)^2$ maps $x_1 \mapsto 0$, $x_2 \mapsto x_4$, $x_3 \mapsto 0$, 
$x_4 \mapsto 0$. However, this map is not
an element of the algebra $\{ \ad x \mid x \in L_{4,3}\}$. Hence we
may assume that $p\geq 3$. Then $(\ad x)^p=0$ for all $x$ and we obtain
that the codomain of the restriction map is contained in the center.
Hence any $[p]$-map  is represented by a vector 
$(\alpha,\beta,\gamma)$ where
$$
\res  x_1=\alpha x_4,\quad \res  x_2=\beta x_4,\quad \res  x_3=\gamma x_4,\quad
\res  x_4=0.
$$

\subsection{Fields of  characteristic  $p\geq 5$}
In this case any $[p]$-map is a semilinear transformation.
Let $\varphi: L\rightarrow \centr L$ be a semilinear transformation and $A\in\aut L$.
Let us compute the vector $(\alpha',\beta',\gamma')$ that determines
$A\varphi A^{-1}$. We have:
\begin{align*}
x_1A\varphi A^{-1}=(a_{11}x_1+a_{12}x_2+a_{13}x_3+a_{14}x_4)\varphi A^{-1}&=
(a_{11}^p\alpha+a_{12}^p\beta+a_{13}^p\gamma)x_4 A^{-1}\\
&=
(a_{11}d_1)^{-1}(a_{11}^p\alpha+a_{12}^p\beta+a_{13}^p\gamma)x_4.
\end{align*}
Hence $\alpha'=(a_{11}d_1)^{-1}(a_{11}^p\alpha+a_{12}^p\beta+a_{13}^p\gamma)$.
We obtain similarly, that 
$\beta'=(a_{11}d_1)^{-1}(a_{22}^p\beta+a_{23}^p\gamma)$ and
%%
%%
%% Further,
%% $$
%% x_3A\varphi A^{-1}=(d_1 x_3+d_2 x_4)\varphi A^{-1}=d_1^p\gamma uA^{-1}=
%% (a_{11}d_1)^{-1}d_1^pu.
%% $$
%%Thus 
%%
$\gamma'=(a_{11}d_1)^{-1}d_1^p$. 
%Thus in matrix form we obtain that
%\begin{equation}\label{act1}
%(\alpha',\beta',\gamma')=(\alpha,\beta,\gamma)
%(a_{11}d_1)^{-1}
%\begin{pmatrix}
%a_{11}^p &0 & 0\\
%a_{12}^p & a_{22}^p & 0\\
%a_{13}^p & a_{23}^p & d_1^p
%\end{pmatrix}.
%\end{equation}
%Let $H$ denote the subgroup of $\gl 3\F$ that contains the matrices of the form
%$$
%\begin{pmatrix}
%a_{11} &0 & 0\\
%a_{12} & a_{22} & 0\\
%a_{13} & a_{23} & d_1
%\end{pmatrix}.
%$$
As above, we let $\varrho$ denote this action of $\aut L$ on $\F^3$.
We need to determine the orbits of vectors $(\alpha,\beta,\gamma)\in\F^3$ 
under the action $\varrho$.
%$$
%(\alpha,\beta,\gamma)(A\varrho )=(a_{11}d_1)^{-1}(\alpha,\beta,\gamma)
%\res A.
%$$
%Since $A\mapsto \res A$ is an automorphism of $H$ with inverse
%$A\mapsto A^{[1/p]}$ it suffices to determine the $H$-orbits under the action
%$$
%(\alpha,\beta,\gamma)(A\varrho_1 )=(a_{11}d_1)^{-1/p}(\alpha,\beta,\gamma)
%A.
%$$
%Let $H_1$ denote the subgroup of $H$ that contains the elements which
%satisfy $a_{11}d_1=1$. Since the map $A\mapsto \res A$ is an automorphism of 
%$H_1$, the orbits of $H_1$ under $\varrho$ coincide with the orbits
%of $H_1$ under the natural action. Thus in what follows we will determine 
%the orbits of $H_1$ under the natural action.

Let $v=(\alpha,\beta,\gamma)\in\F^3$. If $v=(0,0,0)$, 
then $\{v\}$ is clearly an
$\aut L$-orbit. Suppose that $\gamma\neq 0$. Then
$$
(\alpha,\beta,\gamma)\begin{pmatrix}
\gamma^{1/p} & 0 & -\alpha^{1/p} & 0\\
0 & \gamma^{1/p} & -\beta^{1/p} & 0 \\
0 & 0 & \gamma^{2/p} & -\gamma^{1/p}\beta^{1/p}\\
0 & 0 & 0 & \gamma^{3/p}
\end{pmatrix}\varrho=(0,0,\gamma_1)
$$
with $\gamma_1\in\F$. 
Next, if $\gamma\in\F\setminus 0$, then
 $(0,0,\gamma)=(0,0,1)\diag{\gamma^{-1/p},\gamma^{2/p},\gamma^{1/p},1}\varrho$. 
Hence the 
set of vectors $(\alpha,\beta,\gamma)$ with $\gamma\neq 0$ form 
an single orbit with orbit representative $(0,0,1)$. 

Next, if $\gamma=0$, but $\beta\neq 0$, then
$$
(\alpha,\beta,0)
\begin{pmatrix}
1 & -(\alpha/\beta)^{1/p} & 0 & 0\\
0 & 1 & 0 & 0\\
0 & 0 & 1 & 0\\
0 & 0 & 0 & 1
\end{pmatrix}\varrho=(0,\beta,0)
$$
with $\beta_1\in\F$. 

We claim that $(0,\beta_1,0)$ and $(0,\beta_2,0)$ are in the same $\aut L$-orbits
if and only if $\beta_1\beta_2^{-1}$ is a square in $\F$. First assume that
$(0,\beta_1,0)$ and $(0,\beta_2,0)$ are in the same $\aut L$-orbits. Then
$$
(0,\beta_2,0)=(0,\beta_1,0)(A\varrho)
=(a_{11}d_1)^{-1}(\beta_1a_{12}^p,\beta_1a_{22}^p,0).
$$
From this we obtain that $a_{12}=0$ and that 
$\beta_2=(a_{11}d_1)^{-1}a_{22}^p\beta_1$ which gives that 
$\beta_1\beta_2^{-1}=a_{11}^{2}a_{22}^{1-p}$. 
Since $p$ is odd, $\beta_1\beta_2^{-1}\in(\F^*)^2$. 
Assume next that $\beta_1\beta_2^{-1}=\varepsilon^2$ with some
$\varepsilon\in\F$. Then 
$$
(0,\beta_1,0)\diag{\varepsilon,1,\varepsilon,\varepsilon^2}\varrho=
(0,\varepsilon^{-2}\beta_1,0)=(0,\beta_2,0).
$$ 
Hence two elements $(0,\beta_1,0)$ and $(0,\beta_2,0)$ are in the same
orbit if and only if $\beta_1\beta_2^{-1}$ is a square.

Finally, if $\gamma=\beta=0$, but $\alpha\neq 0$, then
$(\alpha,0,0)\diag{1,\alpha,\alpha,\alpha}\varrho=(1,0,0)$. 

Thus we obtain the following elements 
are $\aut L$-orbit representatives:
$(0,0,0)$, $(1,0,0)$, $(0,\beta,0)$, $(0,0,1)$. Further 
$(0,\beta_1,0)$ and $(0,\beta_2,0)$ are in the same
orbit if and only if $\beta_1\beta_2^{-1}$ is a square.

Hence, up to isomorphism,  the possible restricted Lie algebra structures on $L$ are as follows:
\begin{eqnarray*}
L_{4,3}^1&=&\left<x_1,\ x_2,\ x_3,\ x_4\mid [x_1,x_2]=x_3,\ [x_1,x_3]=x_4
\right>;\\
L_{4,3}^2&=&\left<x_1,\ x_2,\ x_3,\ x_4\mid [x_1,x_2]=x_3,\ [x_1,x_3]=x_4,\ 
\res x_1=x_4\right>;\\
L_{4,3}^3(\beta)&=&\left<x_1,\ x_2,\ x_3,\ x_4\mid [x_1,x_2]=x_3,\ [x_1,x_3]=x_4,\ 
\res x_2=\beta x_4\right>;\\
L_{4,3}^4&=&\left<x_1,\ x_2,\ x_3,\ x_4\mid [x_1,x_2]=x_3,\ [x_1,x_3]=x_4,\ 
\res x_3=x_4\right>\\
\end{eqnarray*}

%\begin{eqnarray*}
%L_{4,3}^1&=&\left<x_1,\ x_2,\ x_3,\ x_4\mid [x_1,x_2]=x_3,\ [x_1,x_3]=x_4,\ 
%\res x_1=\res x_2=\res x_3=\res x_4=0\right>;\\
%L_{4,3}^2&=&\left<x_1,\ x_2,\ x_3,\ x_4\mid [x_1,x_2]=x_3,\ [x_1,x_3]=x_4,\ 
%\res x_1=x_4,\ \res x_2=\res x_3=\res x_4=0\right>;\\
%L_{4,3}^3(\beta)&=&\left<x_1,\ x_2,\ x_3,\ x_4\mid [x_1,x_2]=x_3,\ [x_1,x_3]=x_4,\ 
%\res x_2=\beta x_4,\ \res x_1=\res x_3=\res x_4=0\right>;\\
%L_{4,3}^4&=&\left<x_1,\ x_2,\ x_3,\ x_4\mid [x_1,x_2]=x_3,\ [x_1,x_3]=x_4,\ 
%\res x_3=x_4,\ \res x_1=\res x_2=\res x_4=0\right>\\
%\end{eqnarray*}

where $\beta\in\F^*$. 
Further $L_{4,1}^3(\beta_1)\cong L_{4,1}^3(\beta_2)$ if and only if $\beta_1\beta_2^{-1}$ is a square.

\subsection{Fields of characteristic $3$}\label{char3}
In this case we have
\begin{align*}
\resss{(a x_1 +b x_2 +c x_3+dx_4)}&
=a^3\resss x_1 +b^3 \resss x_2+c^3 \resss x_3+a^2b x_4.
\end{align*}

If $\varphi$ is a $[3]$-map then, as in the general case,
 $\varphi$ is represented by a vector 
$(\alpha,\beta,\gamma)$ where
$$
x_1\varphi=\alpha x_4,\quad x_2\varphi=\beta x_4,\quad x_3\varphi=\gamma x_4,\quad
x_4\varphi=0.
$$

Let us compute the vector $(\alpha',\beta',\gamma')$ that determines
$A\varphi A^{-1}$ with $A\in\aut L$:
\begin{align*}
x_1A\varphi A^{-1}&=(a_{11}x_1+a_{12}x_2+a_{13}x_3+a_{14}x_4)\varphi A^{-1}\\
&=(a_{11}^3\alpha+a_{12}^3\beta+a_{13}^3\gamma+a_{11}^2a_{12})x_4 A^{-1}\\
&=(a_{11}d_1)^{-1}(a_{11}^3\alpha+a_{12}^3\beta+a_{13}^3\gamma+a_{11}^2a_{12})x_4.
\end{align*}
Hence $\alpha'=(a_{11}d_1)^{-1}(a_{11}^3\alpha+a_{12}^3\beta+a_{13}^3\gamma+a_{11}^2a_{12})$
and also
$\beta'=(a_{11}d_1)^{-1}(a_{22}^3\beta+a_{23}^3\gamma)$ and 
%% $$
%% zA\varphi A^{-1}=(d_1 x_3+d_2 x_4)\varphi A^{-1}=d_1^3\gamma uA^{-1}=
%%(a_{11}d_1)^{-1}d_1^3u.
%%$$
%%Thus 
$\gamma'=(a_{11}d_1)^{-1}d_1^3=a_{11}^{-1}d_1^2$. 
Thus in matrix form we obtain that the right action $\varrho$ on the set of
$[3]$-maps can be written as
\begin{equation}\label{act2}
(\alpha',\beta',\gamma')=(\alpha,\beta,\gamma)A\varrho=(a_{11}^2a_{22})^{-1}(\alpha,\beta,\gamma)
\begin{pmatrix}
a_{11}^3 &0 & 0\\
a_{12}^3 & a_{22}^3 & 0\\
a_{13}^3 & a_{23}^3 & d_1^3
\end{pmatrix}
+
(a_{22}^{-1}a_{12}, 0, 0).
\end{equation}
Note that $\varrho$ is  not a linear action.
%Let $H$ denote the subgroup of $\gl 3\F$ that contains the matrices of the form
%$$
%\begin{pmatrix}
%a_{11} &0 & 0\\
%a_{12} & a_{22} & 0\\
%a_{13} & a_{23} & a_{11}a_{22}
%\end{pmatrix}.
%$$
%%We need to determine the orbits of vectors $(\alpha,\beta,\gamma)\in\F^3$ 
%%under the action $\varrho$ of $H$ given by
%%$$
%%(\alpha,\beta,\gamma)(A\varrho )=(a_{11}^2a_{22})^{-1}(\alpha,\beta,\gamma)
%%\res A
%%+
%%(a_{11}^2a_{22})^{-1}(a_{11}^2a_{12}, 0, 0).
%%$$
%
%Let $\varrho:H \to M_3(\F)$ given by 
%$$
%(\alpha,\beta,\gamma)A\varrho= (a_{11}^2a_{22})^{-1/p}(\alpha,\beta,\gamma)A +
%((a_{12}/a_{22})^{1/p} , 0, 0),
%$$
%Since $A\mapsto \res A$ is an automorphism of $H$ with inverse
%$A\mapsto A^{[1/p]}$, then $(\alpha,\beta,\gamma)$ and $(\alpha,\beta,\gamma)A\varrho$ represent the same restricted Lie algebra, for every $A\in H$.
%% It is easy to see that $(0, 0, 0)$ and $(\alpha, 0, 0)$ represent the same restricted Lie algebra, for every $\alpha\in \F$. Indeed,
%% $$
%% (0, 0, 0)\begin{pmatrix}
%% 1 & 0 & 0\\
%% \alpha^3 & 1 & 0\\
%% 0 & 0 & 1
%% \end{pmatrix}\varrho
%% =(\alpha,0, 0).
%% $$
Let $v=(\alpha,\beta,\gamma)\in\F^3$.  Suppose that $\gamma\neq 0$. Then
$$
(\alpha,\beta,\gamma)\begin{pmatrix}
\gamma^{1/3} & 0 & -\alpha^{1/3} & 0\\
0 & \gamma^{1/3} & -\beta^{1/3} & 0 \\
0 & 0 & \gamma^{2/3} & -\gamma^{1/3}\beta^{1/3}\\
0 & 0 & 0 & \gamma
\end{pmatrix}\varrho
=(0,0, \gamma^2),
$$
and 
$$
(0,0, \gamma^2)\diag{ \gamma^{-2/3}, \gamma^{-2/3}, \gamma^{-4/3},\gamma^{-2}}\varrho
=(0,0, 1).
$$ 
%$$
Hence the vectors $(\alpha,\beta,\gamma)$, with $\gamma\neq 0$, represent the same restricted Lie algebra as  $(0,0,1)$. 

We notice that the if $A\in \aut L$ as above then 
$(0,0,0)A\varrho=(a_{12}a_{22}^{-1},0,0)$ and hence the $[p]$-maps represented 
by the vectors of the form $(\alpha,0,0)$ form a single orbit with 
orbit representative $(0,0,0)$.

It remains to describe the 
orbits of the $[p]$-maps that are represented by the vectors of 
the form $(\alpha,\beta,0)$ with $\beta\neq 0$. 
For $\beta\in\F^*$, let $\K_{\beta}$ denote the set $\{ \beta x^3+x \mid x \in \F \}$.
Then $\K_{\beta}$ is an $\F_3$-subspace, but it depends on $\F$ and on $\beta$.
For instance if $\F$ is finite and $\beta x^3+x$ has a solution other than 0,
then it has codimension 1, otherwise $\K_{\beta}=\F$.

\begin{lemma}
Let $\alpha_1, \alpha_2\in\F$ and  $\beta_1, \beta_2\in\F^*$. Then
$(\alpha_1,\beta_1,0)$ and $(\alpha_2,\beta_2,0)$ represent
isomorphic restricted Lie algebras if and only if $\beta_1\beta_2^{-1}$ is a
square and $\alpha_2\sqrt{\beta_2/\beta_1}+\alpha_1\in \K_{\beta_1}$ or
$\alpha_2\sqrt{\beta_2/\beta_1}-\alpha_1\in \K_{\beta_1}$.
\end{lemma}
\begin{proof}
The vectors $(\alpha_1,\beta_1,0)$ and $(\alpha_2,\beta_2,0)$ represent
isomorphic restricted Lie algebras if and only if there exists  $A\in \aut L$ such that $(\alpha_1,\beta_1,0)A\varrho=(\alpha_2,\beta_2,0)$. That is  
\begin{align}\label{char 3 orbits}
(\alpha_2,\beta_2,0)=(\alpha_1,\beta_1,0)A\varrho=a_{11}^{-2}a_{22}^{-1}(a_{11}^3\alpha_1+a_{12}^3\beta_1+a_{11}^2a_{12},a_{22}^3\beta_1,0).
\end{align}
Now we deduce  that $\beta_2=a_{11}^{-2}a_{22}^2\beta_1$. In particular
$\beta_2/\beta_1$ is a square. Furthermore, as 
$a_{11}a_{22}^{-1}=\pm\sqrt{\beta_1/\beta_2}$ (meaning that the equality 
holds with plus or minus), 
\begin{multline*}
\alpha_2=\pm a_{11}^{-3}\sqrt{\beta_1/\beta_2}(a_{11}^3\alpha_1+a_{12}^3\beta_1+a_{11}^2a_{12})=\\
\pm(\sqrt{\beta_1/\beta_2}\alpha_1+\sqrt{\beta_1/\beta_2}(a_{11}^{-1}a_{12})^3\beta_1+\sqrt{\beta_1/\beta_2}a_{11}^{-1}a_{12}).
\end{multline*}
Thus
$$
\sqrt{\beta_2/\beta_1}\alpha_2=\pm(\alpha_1+(a_{11}^{-1}a_{12})^3\beta_1+a_{11}^{-1}a_{12}).
$$
Therefore $\sqrt{\beta_2/\beta_1}\alpha_2\pm\alpha_1\in \K_{\beta_1}$,  as
claimed.

Conversely, suppose that  $\alpha_2\sqrt{\beta_2/\beta_1}+\alpha_1\in \K_{\beta_1}$
or  $\alpha_2\sqrt{\beta_2/\beta_1}-\alpha_1\in \K_{\beta_1}$. Choose
$\delta\in\F$ such that, in the latter case, 
$\sqrt{\beta_2/\beta_1}\alpha_2-\alpha_1=\beta_1\delta^3+\delta$,
while 
$-\sqrt{\beta_2/\beta_1}\alpha_2-\alpha_1=\beta_1\delta^3+\delta$ in 
the former.
Then apply
the automorphism
$$
A=\begin{pmatrix}
1 & \delta & 0 & 0\\
0 & \mp\sqrt{\beta_2/\beta_1} & 0 & 0\\
0 & 0 & \mp\sqrt{\beta_2/\beta_1} & 0\\
0 & 0 & 0 &  \mp\sqrt{\beta_2/\beta_1}
\end{pmatrix}
$$
with the sign of $\sqrt{\beta_2/\beta_1}$ chosen accordingly.
It is straight to see that $(\alpha_1,\beta_1,0)A\varrho=(a_2,\beta_2,0)$. 
\end{proof}

Hence, up to isomorphism,  we have the following restricted Lie 
algebras with the underlying Lie algebra $L_{4,3}$ over a field of 
characteristic three:
\begin{eqnarray*}
K_{4,3}^1&=&\left<x_1,\ x_2,\ x_3,\ x_4\mid [x_1,x_2]=x_3,\ [x_1,x_3]=x_4\right>;\\
K_{4,3}^2&=&\left<x_1,\ x_2,\ x_3,\ x_4\mid [x_1,x_2]=x_3,\ [x_1,x_3]=x_4,\ \resss x_3=x_4\right>;\\
K_{4,3}^3(\alpha,\beta)&=&\left<x_1,\ x_2,\ x_3,\ x_4\mid [x_1,x_2]=x_3,\ [x_1,x_3]=x_4,\ 
\resss x_1=\alpha x_4,\ \resss x_2=\beta x_4\right>
\end{eqnarray*}
where $\alpha\in\F$ and $\beta\in\F^*$. 
Furthermore, $K_{4,3}^2(\alpha_1,\beta_1)\cong K_{4,3}^2(\alpha_2,\beta_2)$ 
if and only if $\beta_1\beta_2^{-1}$ is a
square and $\alpha_2\sqrt{\beta_2/\beta_1}+\alpha_1\in \K_{\beta_1}$ or
$\alpha_2\sqrt{\beta_2/\beta_1}-\alpha_1\in \K_{\beta_1}$.
The arguments in the section show that these algebras are pairwise 
non-isomorphic.

\subsection{Characteristic 2} 
As noted at the beginning of the section, in characteristic 2 the Lie 
algebra is not restrictable.

\end{document}